\documentclass[11pt]{article}
\usepackage{authblk}
\usepackage[utf8]{inputenc}
\usepackage{hyperref}
\hypersetup{
     colorlinks = true,
     linkcolor = magenta,
     anchorcolor = magenta,
     citecolor = magenta,
     filecolor = magenta,
     urlcolor = magenta
     }

\usepackage{geometry}
\usepackage{graphicx}
\usepackage{amssymb}
\usepackage{amsthm, amsmath}
\usepackage[all,cmtip]{xy}
\usepackage{color}

\newcommand{\n}{\textbf{n}}
\newcommand{\Z}{ {\mathbb{Z}} } 
\newcommand{\Q}{ {\mathbb{Q}} }
\newcommand{\R}{ {\mathbb{R}} }

\newcommand{\Qp}{ {\mathbb{Q}_p} }

\newcommand{\A}{ \textbf{A} }
\newcommand{\Qv}{ {\mathbb{Q}_v} }

\newcommand{\im}{\operatorname{Im}}

\newtheorem{theorem}{Theorem}[section]

\newtheorem{lemma}[theorem]{Lemma}

\newtheorem{remark}[theorem]{Remark}
\newtheorem{corollary}[theorem]{Corollary}

\title{The size of $2$-Selmer groups for the $\frac{\pi}{3}$-congruent number problem}

\author[1, 2]{Kushal Bhowmick}
\author[1,2]{Aprameyo Pal}

\affil[1]{\raggedright Harish-Chandra Research Institute, Chhatnag Road, Jhunsi, Prayagraj 211 019, India}

\affil[2]{\raggedright Homi Bhabha National Institute, Training School Complex, Anushakti Nagar, Mumbai \newline 400 094, India \newline Email: kushalbhowmick@hri.res.in}

\date{}
\begin{document}

\maketitle

\begin{abstract}
Our main objective in this paper is to study the average rank of the $2$-Selmer group of the elliptic curve associated with the $\frac{\pi}{3}$-congruent number problem. Following Heath-Brown's strategy \cite{HeathBrown},  we could find an asymptotic formula for the size of the relaxed $2$-Selmer groups, which have several consequences towards the average of $2$-Selmer ranks and $\frac{\pi}{3}$-congruent number problem. Indeed, we could find an unconditional positive density of $2$-Selmer rank being $1$ or $3$, among the positive square-free integers $n\equiv 13\pmod{24}$ having all the prime divisors congruent to $1$ modulo $4$ and an unconditional positive density of $2$-Selmer rank being $0$ or $2$, among the positive square-free integers $n\equiv 5\pmod{24}$ having all the prime divisors congruent to $1$ modulo $4$.
\end{abstract}

\noindent \textbf{Keywords:}
Elliptic curves; $\theta$-Congruent numbers; $2$-Selmer rank.

\medskip

\noindent \textbf{MSC 2020:}
11G40; 11G05; 11L40.

\section{Introduction}

\noindent In recent years, the study of quadratic twists of elliptic curves has made significant progress. A key conjecture in this field is Goldfeld's conjecture, which posits that approximately $50\%$ of the quadratic twists of a given elliptic curve have an analytic rank of $1$,  and $50\%$ have an analytic rank of $0$. One of the early breakthroughs towards the conjecture was due to Heath-Brown \cite{HeathBrown}, in the case of congruent elliptic curves $E_n: y^2= x^3-n^2x$.   
\\

\noindent A positive integer \(n\) is a congruent number if it represents the area of a right triangle with rational side lengths. The Congruent Number Conjecture states that for any \(n \equiv 5, 6, 7 \pmod{8}\), \(n\) is congruent; whereas for \(n \equiv 1, 2, 3 \pmod{8}\), almost all \(n\) are non-congruent. Major contributions towards this include results from Tian-Yuan-Zhang \cite{TYZ17} and A. Smith \cite{Smith17} proving some positive density results towards the conjecture.  In particular, the result of Smith has the remarkable consequence that $100\%$ of square-free integers \(n \equiv 1, 2, 3 \pmod{8}\) are not congruent. Tian-Yuan-Zhang \cite{TYZ17} have used the Waldspurger formula and the generalized Gross--Zagier formula of Yuan-Zhang-Zhang, whereas A. Smith's techniques \cite{Smith17} are more similar (but a significant improvement) to   Heath-Brown \cite{HeathBrown}. An important point to note is that this result is beyond the scope of the results from Kane \cite{Kane}, as the considered elliptic curve $E_{1,\frac{\pi}{3}}$ (defined below) in this article, does not satisfy the technical assumption of Kane i.e. $(1-0)(1-(-3))=4$ is a square.
\\

\noindent Fujiwara {\cite{Fujiwara}} introduced a generalized concept of the congruent number. A positive integer n is called a $\theta$-congruent number if there exists a rational $\theta$-triangle with area $n\sqrt{r^{2}-s^{2}}$, where $\theta$ is a real number with $0<\theta<\pi$, $\cos \theta=\frac{s}{r}$ is rational and $r,s\in \Z$, $\gcd(r,s)=1$.
Here a rational $\theta$-triangle is a triangle with rational sides and an angle $\theta$. 
Notably, a congruent number can be seen as a $(\frac{\pi}{2})$-congruent number. The elliptic curves associated to $\theta$-congruent number problem is defined by \begin{center}
$E_{n,\theta} : y^{2} = x(x+n(r+s))(x-n(r-s))$
\end{center}
where $\cos \theta=\frac{s}{r}$, $r,s\in \mathbb{Z}$, $\gcd(r,s)=1$ and $n\in \mathbb{N}$. There are some important criteria for $\theta$-congruent numbers by Fujiwara.
\begin{theorem}{\cite{Fujiwara}}
Let n be any square-free natural number, $0<\theta<\pi$. Then
\item 1) n is $\theta$-congruent if and only if $E_{n,\theta}$ has a rational point of order greater than $2$.
\item 2) For $n\nmid 6$, n is $\theta$-congruent if and only if the Mordell-Weil group $E_{n,\theta}(\Q)$ has a positive rank.   
\end{theorem}

\noindent In this article we would mostly deal with the case of $\theta = \frac{\pi}{3}$. We will remark, on how the case of $\theta = \frac{2\pi}{3}$ can follow similarly.  Inspired by the arithmetic theory of elliptic curves,  the following conjecture is raised by Yoshida \cite{Yoshida}.
\\

\noindent {\bf Conjecture 1}\label{theta-congruent-conjecture}
Let $n >5$ be a square-free integer. Then \\
(1) $n$ is a $\frac{\pi}{3}$-congruent number if  $n\equiv 6,10, 11, 13, 17, 18, 21, 22, 23 \pmod{24}$.\\
(2) $n$ is a $\frac{2\pi}{3}$-congruent number if  $n \equiv 5, 9, 10, 15, 17, 19, 21, 22, 23 \pmod{24}$.\\

\noindent A number being $\frac{\pi}{3}$-congruent or  $\frac{2\pi}{3}$-congruent is also related to another famous problem of Tiling numbers.  An integer \( n > 1 \) is called a tiling number (TN) if there exists a triangle \( \Delta \) and a positive integer \( k \) such that \( nk^2 \) copies of \( \Delta \) can be tiled into an equilateral triangle.  Fujiwara \cite{Fujiwara} has shown that $n$ is a tiling number if and only if $n$ is a $\frac{\pi}{3}$ or  $\frac{2\pi}{3}$-congruent number. There are some results due to Hibino and Kan \cite{HK} and 
He-Hu-Tian in \cite{HHT21} using Heegner points.  A recent achievement towards this is due to 
Feng-Liu-Pan-Tian \cite{FLPT} using the Waldspurger formula and the induction method as in \cite{TYZ17}.
\\

\noindent In this article, we adopt the strategy used by Heath-Brown \cite{HeathBrown}.  Consider the elliptic curves associated to $\frac{\pi}{3}$-congruent number problem, defined by:
$$E_{n,\frac{\pi}{3}} :  y^{2} = x(x+3n)(x-n)$$     

\noindent where $n\in \mathbb{N}$. One can check that $E_{n,\frac{\pi}{3}}$ has four $2$-torsion points given by:
$$E_{n,\frac{\pi}{3}}[2]= \lbrace \Theta, (0,0), (-3n,0), (n,0)\rbrace.$$
Moreover, one can also show that the $\Q$-torsion points of $E_{n,\frac{\pi}{3}}$ are the same as the $2$-torsion points. As a consequence:
$$E_{n,\frac{\pi}{3}}(\Q)_{tors}= \lbrace \Theta, (0,0), (-3n,0), (n,0)\rbrace \simeq\frac{\Z}{2\Z}\oplus \frac{\Z}{2\Z} \quad \text{for}\quad  n\nmid 6.$$
Via the Mordell-Weil theorem, one obtains:
$$E_{n,\frac{\pi}{3}}(\Q) \simeq \Z^{r(n)}\oplus E_{n,\frac{\pi}{3}}(\Q)_{tors}$$
where $r(n)$ is a non-negative integer, which is called the Mordell-Weil rank of the elliptic curve $E_{n,\frac{\pi}{3}}$. Let $Sel_{2}(E_{n,\frac{\pi}{3}}\diagup \Q)$ denotes the $2$-Selmer group over $\Q$. We know that  $\#\, Sel_{2}(E_{n,\frac{\pi}{3}}\diagup \Q) = 2^k$, where $k \geq 2$, since $E_{n,\frac{\pi}{3}}$ has four rational $2$-torsion points. We shall therefore write $\,\# \, Sel_{2}(E_{n,\frac{\pi}{3}}\diagup \Q) = 2^{2+s(n)}$. The exponent $s(n)$ is called the $2$-Selmer rank of the elliptic curve $E_{n,\frac{\pi}{3}}$. One can deduce from above $\#\: \frac{E_{n,\frac{\pi}{3}}(\Q)}{2E_{n,\frac{\pi}{3}}(\Q)}= 2^{2+r(n)}$. As we know that $\frac{E_{n,\frac{\pi}{3}}(\Q)}{2E_{n,\frac{\pi}{3}}(\Q)}\hookrightarrow Sel_{2}(E_{n,\frac{\pi}{3}}\diagup \Q)$, it implies $r(n)\leq s(n)$.\\

\noindent The aim of this paper is to investigate $s(n)$ on average. The main result proved in the article is a result similar to Heath-Brown for $E_{n,\frac{\pi}{3}}$ involving the $2$-Selmer rank.
\begin{theorem}{\label{maintheorem}}
For $\, h=5,13, \quad$ let
\begin{center}
$S(X,h)= \lbrace1\leq n\leq X :\,n\equiv h \pmod{24}, \,n\,\text{square-free},\, p\equiv 1 \pmod{4} \text{\, for every prime}\: p\mid n\rbrace$.
\end{center}
Then \begin{center} 
$\sum\limits_{n\in S(X,h)}2^{s(n)}= 9\,\#\, S(X,h)+ O(X(\log X)^{-\frac{5}{8}}(\log \log X)^{8})$.
\end{center}
\end{theorem}
\begin{remark}
\leavevmode

\begin{itemize}
\item  We also observe that, based on computations in \texttt{SageMath}, Heath--Brown's technique allows us to handle only integers  $n \equiv 5,13 \pmod{24}$ whose all prime factors satisfy $p \equiv 1 \pmod{4}$. The technical difficulties arising in the case are mentioned in more detail in the beginning of Section  \ref{counting 2-descent}. In particular, any progress beyond this range would necessitate new methods. 
\item One interesting point to observe is that the average rank of the $2$-Selmer group is inconsistent with Bhargava-Kane-Lenstra-Poonen-Rains \cite{BKLPR}.  
\end{itemize}
\end{remark}

\noindent Using the inequality $r(n)\leq s(n)$, one obtains the following. 
\begin{corollary}\label{averagerank}
For  $\, h=5,13, \quad$ we have,
\begin{center}
$\sum\limits_{n\in S(X,h)}2^{r(n)}\leq 9\,\#\, S(X,h)+ O(X(\log X)^{-\frac{5}{8}}(\log \log X)^{8})$.
\end{center}
\end{corollary}

\noindent Using the parity results proved in \cite{WeiGuo} and the main Theorem \ref{maintheorem}, we can deduce several corollaries towards the density statements of $s(n) = 0$ or $2$ (Corollary \ref{s(n)=0or2}) and unconditional density results towards $s(n) = 1$ or $3$ (Corollary \ref{s(n)=1}). These results provide further evidence towards Conjecture \ref{theta-congruent-conjecture}.  These applications are collected in the Section \ref{applications}. 
\\

\noindent 
Another technical aspect that deserves emphasis is the role of local conditions at primes dividing $6n$. In the computation of the $2$-Selmer group, the main obstacle is to understand the image of the Kummer map inside $\Q^{\times}/(\Q^{\times})^{2}$. For the family $E_{n,\frac{\pi}{3}}$, the structure of the model gives rise to a naturally defined $2$-isogeny, and this allows us to describe the local images explicitly in terms of Hilbert symbols and quadratic residue conditions. The assumption that every prime divisor of $n$ satisfies $p \equiv 1 \pmod{4}$ plays a crucial role here, since under this restriction the relevant local groups behave in a much more controlled manner. As a result, the global Selmer group may be expressed in terms of a system of homogeneous linear conditions over $\mathbb{F}_{2}$, and its size reduces to counting the solutions to these systems. From a Galois cohomological viewpoint, the $2$-Selmer group sits inside $H^{1}(\Q, E_{n,\frac{\pi}{3}}[2])$, and since the curve has full rational $2$-torsion, this cohomology group decomposes into three copies of $H^{1}(\Q, \Z/2\Z)$; the Selmer group then appears as the intersection of three affine subspaces cut out by local restriction maps. Consequently, the distribution of $2$-Selmer ranks in this family is governed not by the random matrix heuristics of Poonen-Rains, but by a more rigid structure originating from the $2$-isogeny. The asymptotic constant $9$ in Theorem~\ref{maintheorem} reflects that, on average, the $2$-Selmer condition behaves like the intersection of three independent affine $\mathbb{F}_{2}$-hyperplanes, a phenomenon not predicted by the general heuristics of Bhargava-Kane-Lenstra-Poonen-Rains. Understanding why this rigidity emerges, and whether similar deviations occur in other $\theta$-congruent families, remains an intriguing direction for future research.
\\

\noindent We would like to close the introduction by pointing out some possible future directions. One immediate direction, one could proceed is towards the results related to $\frac{2\pi}{3}$-congruent number problem setting. In this setting, one would also proceed similarly as in this article. This would also lead to some new character sums like in Section \ref{sec:6}.  Another natural direction to follow up will be to find out the exact distribution of $s(n)$ as in Heath-Brown \cite{HeathBrown1}.  This would also start with Theorem \ref{maintheorem} as a crucial input and one has to deal with some extra contributions in the character sum. \\

\noindent {\bf Structure of the article.} The structure of the article is as follows.  In the 2nd section,  first, we find out the homogeneous equations via the $2$-descent method and obtain an alternative expression for the average rank of the $2$-Selmer group.  The expression requires us to take an average over the component variables of a given integer $n$.  The 3rd section takes the average over the linked variables.  In the 4th section,  one needs to take the average over the remaining sum.  Using some known estimates, one could get a bound over characters modulo $24$. The remaining sums are taken care of in the 5th section, which contributes to the leading terms in the statement of the main Theorem \ref{maintheorem}.  Section 6,  lists several immediate applications including unconditional $2$-Selmer ranks and evidences towards $\frac{\pi}{3}$-congruent number problem.\\

\noindent {\bf Acknowledgements.} The first author would like to thank Dr.Ratnadeep Acharya and Dr.Ravi Theja for several useful discussions during the preparation of the article. The second author acknowledges the support of the SERB Matrics grant (MTR/2022/000302) and the SERB SRG grant (SRG/2022/000647). The first author acknowledges support from the HRI Institute fellowship and the Infosys grant of Dr. Jishnu Ray. 

\section{Counting 2-descents}\label{counting 2-descent}
\noindent In this section, we begin with the elliptic curve $E_{n,\frac{\pi}{3}}$ and get an alternative description of $s(n)$ by counting $2$-descents following Heath-Brown \cite{HeathBrown}.  
$$E_{n,\frac{\pi}{3}} :  y^{2} = x(x+3n)(x-n) \qquad n\in \mathbb{N}_{>1},\, \text{square-free},\, (n,24)=1.$$
Let us recall the homomorphism
$$\theta: \frac{E_{n,\frac{\pi}{3}}(\Q)}{2E_{n,\frac{\pi}{3}}(\Q)}\rightarrow Q\times Q\times Q,\quad \text{where}\quad Q=\frac{\Q^{\times}}{\Q^{\times2}}$$
\begin{center}
$(x,y)\mapsto (x,x+3n,x-n)\,(\mathrm{mod} \,\Q^{\times2}).$
\end{center}
Note that the map is defined except at the torsion points. This is a well-known fact that $\theta$ is injective and $E_{n,\frac{\pi}{3}}(\Q)$ has no torsion points other than $2$-torsion points. We now figure out that a coset of $E_{n,\frac{\pi}{3}}(\Q)_{tors}$ in $E_{n,\frac{\pi}{3}}(\Q)$ associated to a non-torsion point, consists of 
\begin{align}\label{coset}
&(x,y),\quad (x,y)+(0,0)=\left(\frac{-3n^{2}}{x},\frac{x+3n^{2}y}{x^{2}}\right),   \quad (x,y)+(n,0)=\left(\frac{n(x+3n)}{x-n},\frac{-4n^{2}y}{(x-n)^{2}}\right), \quad \text{and}\\&(x,y)+(-3n,0)=\left(\frac{-3n(x-n)}{x+3n},\frac{-12n^{2}y}{(x+3n)^{2}}\right) \notag.
\end{align} 

\noindent Using \texttt{SageMath} calculations, we observe that for $n \equiv 1, 17 \pmod{24} $, there is no point $(x,y)$ in the above coset that satisfies $x>0$ and $|x|_{2}\neq 1$. To obtain an equation as in \cite{HeathBrown}, we require a unique point $(x,y)$ in the above coset satisfying $x>0$, $|x|_{2}\neq 1$ and either $v_{3}(x) \leq 0$ or $v_{3}(x)> 0$.
Using \texttt{SageMath} calculations, we could not find any point $(x,y)$ in the above coset satisfying $x>0$, $|x|_{2}\neq 1$ and  $v_{3}(x) \leq 0$ (respectively $x>0$, $|x|_{2}\neq 1$ and $v_{3}(x)> 0$) for all $n \equiv 7,11,23 \pmod{24}$. Finally we shall prove that there is a unique point $(x,y)$ in the above coset satisfying $x>0$, $|x|_{2}\neq 1$ and $v_{3}(x) \leq 0$, whenever $n \equiv 5,13 \pmod{24}$ and every prime divisor of $n$ is congruent to $1$ modulo $4$. By \texttt{SageMath}, we found that the above assertion does not hold if we drop the condition on the prime divisors of $n$.\\

\noindent From now on we assume that $n\in S(X,h)$ and $h=5,13$.
\begin{lemma}\label{important lemma}
There is a unique point $(x,y)$ in the coset (\ref{coset}) satisfying $x>0$ and $|x|_{2}\neq 1$.    
\end{lemma}
\begin{proof}
If $|x|_{2}\neq 1$, then we have $|\frac{n(x+3n)}{x-n}|_{2}=1$ and $|\frac{-3n(x-n)}{x+3n}|_{2}=1$. Therefore we get a unique point in the coset (\ref{coset}) satisfying the above condition. 

\noindent Now we assume that $|x|_{2}=1$. Note that it suffices to show $|x+3n|_{2}\neq |x-n|_{2}$. As $n$, $3n$ are odd and $x\equiv 1\pmod{2}$ we get
$$x-n=2x^{\prime}, \quad x+3n=2x^{\prime \prime} \quad \mathrm{for\: some} \quad x^{\prime},x^{\prime \prime}\in \mathbb{Z}_{2}.$$ Thus $2n=x^{\prime \prime}-x^{\prime}$. Since $n\equiv 5\pmod{8}$ we obtain
$$x^{\prime \prime}-x^{\prime}\equiv 2\pmod{8}.$$
If $x^{\prime\prime}\equiv 0,2,4,6\pmod{8}$, then we are done. So  assume that $x^{\prime\prime}\equiv 1,3,5,7\pmod{8}$. This implies that $$x\equiv 3,7\pmod{8}.$$
Below we show that this leads to a contradiction.

\noindent Now we write $\, x=\frac{r}{s},\quad y=\frac{t}{u}$ with $(r,s)=(t,u)=1$ and where $r$, $s$ are odd. Then
\begin{align*}
\frac{t^{2}}{u^{2}}  & =\frac{r}{s}(\frac{r}{s}+3n)(\frac{r}{s}-n), \\ 
\Rightarrow s^{3}t^{2} &  =r(r+3sn)(r-sn)u^{2}.
\end{align*}
Now $u^{2}\mid s^{3}$, since $t$ and $u$ are coprime. Similarly we observe that $(s^{3}, r(r+3sn)(r-sn))=1$ as $r$ and $s$ are coprime, so we have $s^{3}\mid u^{2}$. Therefore $s^{3}=u^{2}$, and it implies that $s=W^{2}$, $u=W^{3}$ for some $W \in \Z$. We now have $$r(r+3sn)(r-sn)=t^{2}.$$ 
As $r+3sn$, $r-sn$ are both even, we get $2\mid t$. So we have
\begin{equation}\label{equn}
r\left(\frac{r+3sn}{2}\right)\left(\frac{r-sn}{2}\right)=\left(\frac{t}{2}\right)^{2}.    
\end{equation}
Case 1)  $x\equiv 3\pmod{8}$. Since $s=W^{2}$, we have $s\equiv 1\pmod{8}$. Therefore $r\equiv 3\pmod{8}$ and $\frac{r-sn}{2}\equiv 3\pmod{4}$. We write $\frac{r-sn}{2}=4k+3$. Therefore $\frac{r+3sn}{2}=4k+3+2sn$. Also note that
$$\frac{r+3sn}{2}-\frac{r-sn}{2}=2sn\equiv 2\pmod{8}.$$
Putting all these together we get
$$r\left(\frac{r+3sn}{2}\right)\left(\frac{r-sn}{2}\right)=r(4k+3+2ns)(4k+3)\equiv 5\pmod{8}.$$
Using this in (\ref{equn}) we obtain a contradiction.\\

\noindent Case 2)  $x\equiv 7\pmod{8}$. Since $s=W^{2}$, we have $s\equiv 1\pmod{8}$. Therefore $r\equiv 7\pmod{8}$ and $\frac{r-sn}{2}\equiv 1\pmod{4}$. We write  $\frac{r-sn}{2}=4k+1$. Therefore $\frac{r+3sn}{2}=4k+1+2sn$. We obtain $\frac{r+3sn}{2}$, $\frac{r-sn}{2}$ are both odd. As in the previous case, we have $2sn\equiv 2\pmod{8}$. Therefore
$$r\left(\frac{r+3sn}{2}\right)\left(\frac{r-sn}{2}\right)=r(4k+1+2ns)(4k+1)\equiv 5\pmod{8}$$
which is again a contradiction.\\
\noindent Hence proved.
\end{proof}

\begin{lemma}\label{lmma}
Suppose $(x,y)\in E_{n,\frac{\pi}{3}}(\Q)\setminus  E_{n,\frac{\pi}{3}}(\Q)_{tors}$ be such that $x>0$ and $|x|_{2}\neq 1$. Then $v_{3}(x)\leq 0$.
\end{lemma}
\begin{proof}
Let us assume that $v_{3}(x)>0$. we write $\, x=\frac{r}{s},\quad y=\frac{t}{u}$ with $(r,s)=(t,u)=1$ and where $r,s,u >0$ and $r,s$ have opposite parities and $3\mid r$. Then
\begin{align*}
\frac{t^{2}}{u^{2}}  & =\frac{r}{s}(\frac{r}{s}+3n)(\frac{r}{s}-n), \\ 
\Rightarrow s^{3}t^{2} &  =r(r+3sn)(r-sn)u^{2}.
\end{align*}
Now $u^{2}\mid s^{3}$, since $t$ and $u$ are coprime. Similarly we observe that $(s^{3}, r(r+3sn)(r-sn))=1$ as $r$ and $s$ are coprime, so we have $s^{3}\mid u^{2}$. Therefore $s^{3}=u^{2}$, and it implies that $s=W^{2}$, $u=W^{3}$ for some $W \in \Z$. We now have
$$r(r+3sn)(r-sn)=t^{2}.$$ 
Let's denote $(r,n)=n_{0}$ and $r=n_{0}r^{\prime}$, 
$$n_{0}r^{\prime}(n_{0}r^{\prime}+3sn)(n_{0}r^{\prime}-sn)=t^{2},$$
$$\Rightarrow n_{0}^{3}r^{\prime}(r^{\prime}+3s\frac{n}{n_{0}})(r^{\prime}-s\frac{n}{n_{0}})=t^{2},$$
$\Rightarrow n_{0}^{3}\mid t^{2}$. Then $n_{0}^{2}\mid t$ Since $n_{0}$ is square-free. Thus we obtain, 
$$r^{\prime}(r^{\prime}+3s\frac{n}{n_{0}})(r^{\prime}-s\frac{n}{n_{0}})=n_{0}\left( \frac{t}{n_{0}^{2}} \right)^{2}.$$
We write $r^{\prime}=3r^{\prime\prime}$. Thus we get
$$9r^{\prime\prime}(r^{\prime\prime}+s\frac{n}{n_{0}})(3r^{\prime\prime}-s\frac{n}{n_{0}})=n_{0}\left( \frac{t}{n_{0}^{2}} \right)^{2}.$$
We observe that $3\mid t$. Therefore we obtain
$$r^{\prime\prime}(r^{\prime\prime}+s\frac{n}{n_{0}})(3r^{\prime\prime}-s\frac{n}{n_{0}})=n_{0}\left( \frac{t}{3n_{0}^{2}} \right)^{2}.$$
As $n$ and $r+s$, are odd, it implies that  $r^{\prime\prime}+s\frac{n}{n_{0}}$ is odd. Since we know that $(r^{\prime\prime},s\frac{n}{n_{0}})=1$, then we conclude $r^{\prime\prime}, (r^{\prime\prime}+s\frac{n}{n_{0}}), (3r^{\prime\prime}-s\frac{n}{n_{0}})$ are pairwise coprime. Therefore we may write
$$n_{0}=n_{1}n_{2}n_{3},\quad \frac{t}{3n_{0}^{2}}=XYZ,$$
and
$$r^{\prime\prime}=n_{1}X^{2},\quad r^{\prime\prime}+s\frac{n}{n_{0}}=n_{2}Y^{2},\quad 3r^{\prime\prime}-s\frac{n}{n_{0}}=n_{3}Z^{2}.$$
Therefore from above we obtain
$$n_{2}+n_{3}\equiv 0\pmod{4}.$$
Which is a contradiction to all prime divisors of $n$ are congruent to $1$ modulo $4$.
Hence we are done.
\end{proof}

\noindent From Lemma \ref{important lemma} and Lemma \ref{lmma} we obtain that there is a unique point $(x,y)$ in the coset (\ref{coset}) with $x>0$, $|x|_{2}\neq 1$ and $v_{3}(x)\leq 0$. Therefore $\im(\theta)$ restricted to all the points $(x,y)$ with $x>0$, $|x|_{2}\neq 1$ and $v_{3}(x)\leq 0$ will have size $2^{r(n)}$.\\

\noindent To analyze $\im(\theta)$, we write $\, x=\frac{r}{s},\quad y=\frac{t}{u}$ with $(r,s)=(t,u)=1$ and where $r,s,u >0$ and $r,s$ have opposite parities and $3\nmid r$. Then
\begin{align*}
\frac{t^{2}}{u^{2}}  & =\frac{r}{s}(\frac{r}{s}+3n)(\frac{r}{s}-n), \\ 
\Rightarrow s^{3}t^{2} &  =r(r+3sn)(r-sn)u^{2}.
\end{align*}
Now $u^{2}\mid s^{3}$, since $t$ and $u$ are coprime. Similarly we observe that $(s^{3}, r(r+3sn)(r-sn))=1$ as $r$ and $s$ are coprime, so we have $s^{3}\mid u^{2}$. Therefore $s^{3}=u^{2}$, and it implies that $s=W^{2}$, $u=W^{3}$ for some $W \in \Z$. We now have
$$r(r+3sn)(r-sn)=t^{2}.$$ 
Let's denote $(r,n)=n_{0}$ and $r=n_{0}r^{\prime}$, 
\begin{center}
$n_{0}r^{\prime}(n_{0}r^{\prime}+3sn)(n_{0}r^{\prime}-sn)=t^{2}$\\
\vspace*{0.3cm}
$\Rightarrow n_{0}^{3}r^{\prime}(r^{\prime}+3s\frac{n}{n_{0}})(r^{\prime}-s\frac{n}{n_{0}})=t^{2}$,
\end{center}
$\Rightarrow n_{0}^{3}\mid t^{2}$. Then $n_{0}^{2}\mid t$ Since $n_{0}$ is square-free. Thus we obtain, 
\begin{equation}\label{eq:1}
r^{\prime}(r^{\prime}+3s\frac{n}{n_{0}})(r^{\prime}-s\frac{n}{n_{0}})=n_{0}\left( \frac{t}{n_{0}^{2}} \right)^{2}.
\end{equation}
As $n$ and $r+s$, are odd, it implies that  $r^{\prime}+s\frac{n}{n_{0}}$ is odd. Since we know that $(r^{\prime},s\frac{n}{n_{0}})=1$ and $3\nmid r^{\prime}$, then we conclude $r^{\prime}, (r^{\prime}+3s\frac{n}{n_{0}}), (r^{\prime}-s\frac{n}{n_{0}})$ are pairwise coprime. Therefore we may write
$$n_{0}=n_{1}n_{2}n_{3},\quad tn_{0}^{-2}=XYZ,$$
and
$$r^{\prime}=n_{1}X^{2},\quad r^{\prime}+3s\frac{n}{n_{0}}=n_{2}Y^{2},\quad r^{\prime}-s\frac{n}{n_{0}}=n_{3}Z^{2}.$$
\noindent By taking $\frac{n}{n_{0}}=n_{4}$, we get the system of equations,
\begin{equation}\label{eq:2}
n_{1}X^{2}+3n_{4}W^{2}=n_{2}Y^{2},\quad n_{1}X^{2}-n_{4}W^{2}=n_{3}Z^{2}.
\end{equation}
\noindent $r>0$ and $n_{0}>0$ implies that $n_{1},n_{4}>0$ and we note that if the first of the equations (\ref{eq:2}) has non-trivial real solutions then $n_{2}>0$. Then as $n=n_{1}n_{2}n_{3}n_{4}>0$, we note that $n_{3}>0$. Therefore we have obtained the following lemma.
\begin{lemma}\label{lemma1}
There are $2^{r(n)}$ systems (\ref{eq:2}) with non-trivial integer solutions. Moreover from the definition of $s(n)$, it is evident that there are $2^{s(n)}$ systems (\ref{eq:2}) which are solvable in $\mathbb{R}$ and in $\Qp$ for all prime $p$. 
\end{lemma}
\noindent The conditions of $n_{j}>0$ for $j=1,2,3,4$ ensures that (\ref{eq:2}) have real soultions. Moreover we can show that (\ref{eq:2}) has non-trivial solutions in $\Qp$ whenever $p\nmid 6n$. Consequently, we have the following remaining conditions: \\

\noindent For $p\mid n_{1}$, (\ref{eq:2}) has non-trivial solutions in $\Qp$ if and only if $ \left(\frac{3n_{2}n_{4}}{p}\right) = \left(\frac{-n_{3}n_{4}}{p}\right)= 1$\\
For $p\mid n_{4}$, (\ref{eq:2}) has non-trivial solutions in $\Qp$ if and only if $ \left(\frac{n_{1}n_{2}}{p}\right) = \left(\frac{n_{1}n_{3}}{p}\right)= 1$\\
For $p\mid n_{2}$, (\ref{eq:2}) has non-trivial solutions in $\Qp$ if and only if $ \left(\frac{-3n_{1}n_{4}}{p}\right) = \left(\frac{-n_{3}n_{4}}{p}\right)= 1$\\
For $p\mid n_{3}$, (\ref{eq:2}) has non-trivial solutions in $\Qp$ if and only if $ \left(\frac{n_{1}n_{4}}{p}\right) = \left(\frac{n_{2}n_{4}}{p}\right)= 1$\\

\noindent Via a careful case-by-case study, we obtain the following conclusions. 
\begin{lemma}\label{Q_2localcondn}
The system  (\ref{eq:2}) is solvable in $\mathbb{Q}_{2}$ if and only if one of the following set of conditions holds:
\begin{align*}
i) \quad &n_{2}\equiv 3n_{4}\pmod{4}, n_{3}\equiv -n_{4}\pmod{4}, n_{2}+3n_{3}\equiv 0\pmod{8}\\
ii) \quad &n_{1}\equiv n_{2}\equiv n_{3}\pmod{4}, n_{2}\equiv n_{3}\pmod{8}\\
iii) \quad &n_{1}\equiv 3n_{3}\pmod{4}, n_{4}\equiv -n_{3}\pmod{4}, n_{1}+3n_{4}\equiv 0\pmod{8}\\
iv) \quad &n_{2}\equiv n_{1}\equiv n_{4}\pmod{4}, n_{1}\equiv n_{4}\pmod{8}
\end{align*}
\end{lemma}
\noindent We can easily observe that in our case either ii) or iv) is satisfied as $n\in S(X,h)$ and $h=5,13$. Therefore by Lemma \ref{Q_2localcondn} we get System (\ref{eq:2}) is solvable in  $\mathbb{Q}_{2}$.
\begin{lemma}\label{Q_3localcondn}
The system  (\ref{eq:2}) is solvable in $\mathbb{Q}_{3}$ if and only if one of the following set of conditions holds:
\begin{align*}
i) \quad &n_{1}\equiv n_{2}\equiv n_{3}\pmod{3}\\
ii) \quad &n_{1}\equiv  n_{2}\equiv n_{4}\pmod{3}\\
iii) \quad &n_{1}\equiv  n_{2}\pmod{3}, n_{1}-n_{4}\equiv n_{3}\pmod{3}
\end{align*}
\end{lemma}
\begin{lemma}
If the system (\ref{eq:2}) is solvable in $\R$ and in $\Qp$ for all prime $p\neq 3$, then it is solvable in $\mathbb{Q}_{3}$ .
\end{lemma}
\begin{proof}
By the product formula for the Hasse norm residue symbol, we observe that the equation
$$n_{1}X^{2}+3n_{4}W^{2}=n_{2}Y^{2}$$
has solutions in $\mathbb{\Q}_{3}$. We may assume that such a solution involves $3$-adic integers, $Y$ is a $3$-adic unit. Therefore from the above equation we can conclude that $X$ is also a $3$-adic unit. Hence we have 
$$X^{2}\equiv Y^{2}\equiv 1\pmod{3}.$$
Therefore from the above equation we have
\begin{equation}\label{important equation}
n_{1}\equiv n_{2}\pmod{3}.  
\end{equation}
By the same argument, we can conclude that the equation
$$n_{1}X^{2}-n_{4}W^{2}=n_{3}Z^{2}$$
has solutions in $\mathbb{\Q}_{3}$. We may assume that such a solution involves $3$-adic integers, at least one of them is a unit. Therefore from the above equation we can conclude that at least two of them are units.\\

\noindent Case 1) $X\in \mathbb{Z}_{3}^\times$, $W\in \mathbb{Z}_{3}^\times$, $3\mid Z$. Therefore we have
$$n_{1}\equiv n_{4}\pmod{3}.$$
\noindent Case 2) $X\in \mathbb{Z}_{3}^\times$, $Z\in \mathbb{Z}_{3}^\times$, $3\mid W$. Therefore we have
$$n_{1}\equiv n_{3}\pmod{3}.$$
\noindent Case 3) $W\in \mathbb{Z}_{3}^\times$, $Z\in \mathbb{Z}_{3}^\times$, $3\mid X$. Therefore we have
$$n_{3}+n_{4}\equiv 0\pmod{3}.$$
\noindent Case 4) $X\in \mathbb{Z}_{3}^\times$, $W\in \mathbb{Z}_{3}^\times$, $Z\in \mathbb{Z}_{3}^\times$. Therefore we have
$$n_{1}-n_{4}\equiv n_{3}\pmod{3}.$$
We can observe that (\ref{important equation}) and Case 3) implies iii) in Lemma \ref{Q_3localcondn}. Therefore from Lemma  \ref{Q_3localcondn}, we can conlcude that the system (\ref{eq:2}) is solvable in $\mathbb{\Q}_{3}$.
\end{proof}

\noindent Now we are in a position to write down our formula for $2^{s(n)}$.
We define
\begin{center}
$\qquad \Pi_{1}=\prod\limits_{p\mid n_{1}}\biggl\{1+\biggl(\frac{3n_{2}n_{4}}{p}\biggr)+\biggl(\frac{-n_{3}n_{4}}{p}\biggr)+\biggl(\frac{-3n_{2}n_{3}}{p}\biggr)\biggr\}$    
\end{center}
\begin{center}
$\qquad \Pi_{2}=\prod\limits_{p\mid n_{2}}\biggl\{1+\biggl(\frac{-3n_{1}n_{4}}{p}\biggr)+\biggl(\frac{-n_{3}n_{4}}{p}\biggr)+\biggl(\frac{3n_{1}n_{3}}{p}\biggr)\biggr\}$    
\end{center}
\begin{center}
$\Pi_{3}=\prod\limits_{p\mid n_{3}}\biggl\{1+\biggl(\frac{n_{1}n_{4}}{p}\biggr)+\biggl(\frac{n_{2}n_{4}}{p}\biggr)+\biggl(\frac{n_{1}n_{2}}{p}\biggr)\biggr\}$    
\end{center}
\begin{center}
$\Pi_{4}=\prod\limits_{p\mid n_{4}}\biggl\{1+\biggl(\frac{n_{1}n_{2}}{p}\biggr)+\biggl(\frac{n_{1}n_{3}}{p}\biggr)+\biggl(\frac{n_{2}n_{3}}{p}\biggr)\biggr\}$    
\end{center}
We see that the expression
$$4^{-\omega(n)}\Pi_{1}\Pi_{2}\Pi_{3}\Pi_{4},$$
where $\omega(n)=\#\:\left\{p\: \text{prime}: p\mid n \right\}$, will be $1$ if the system (\ref{eq:2}) is solvable in $\Qv, \forall v\in M_{\Q}$, and $0$ otherwise. We can now expand $\Pi_{1}$, $\Pi_{2}$, $\Pi_{3}$, $\Pi_{4}$ as follows:
$$\qquad \Pi_{1}=\sum \biggl(\frac{3n_{2}n_{4}}{n_{13}}\biggr)\biggl(\frac{-n_{3}n_{4}}{n_{12}}\biggr)\biggl(\frac{-3n_{2}n_{3}}{n_{14}}\biggr),$$
here we take the sum over all factorizations, 
$n_{1}=n_{10}n_{12}n_{13}n_{14}$, $n_{10}$, $n_{12}$, $n_{13}$, $n_{14}$ are pairwise coprime.
$$\qquad \Pi_{2}=\sum \biggl(\frac{-3n_{1}n_{4}}{n_{23}}\biggr)\biggl(\frac{-n_{3}n_{4}}{n_{21}}\biggr)\biggl(\frac{3n_{1}n_{3}}{n_{24}}\biggr),$$
here we take the sum over all factorizations, $n_{2}=n_{20}n_{21}n_{23}n_{24}$, $n_{20}$, $n_{21}$, $n_{23}$, $n_{24}$ are pairwise coprime.
$$\Pi_{3}=\sum \biggl(\frac{n_{1}n_{4}}{n_{32}}\biggr)\biggl(\frac{n_{2}n_{4}}{n_{31}}\biggr)\biggl(\frac{n_{1}n_{2}}{n_{34}}\biggr),$$
here we take the sum over all factorizations, $n_{3}=n_{30}n_{31}n_{32}n_{34}$, $n_{30}$, $n_{31}$, $n_{32}$, $n_{34}$ are pairwise coprime.
$$\Pi_{4}=\sum \biggl(\frac{n_{1}n_{2}}{n_{43}}\biggr)\biggl(\frac{n_{1}n_{3}}{n_{42}}\biggr)\biggl(\frac{n_{2}n_{3}}{n_{41}}\bigg),$$
where we take the sum over all factorizations, $n_{4}=n_{40}n_{41}n_{42}n_{43}$, $n_{40}$, $n_{41}$, $n_{42}$, $n_{43}$ are pairwise coprime.\\

\noindent Here $\: \n$ represents the 16-tuple of elements $n_{ij}$ with $1\leq i\leq 4$, $0\leq j\leq 4$ and $i\neq j$. According to Lemma \ref{lemma1}, sum of the expression $4^{-\omega(n)}\Pi_{1}\Pi_{2}\Pi_{3}\Pi_{4}$ over all quadruples $n_{1}$, $n_{2}$, $n_{3}$, $n_{4}$ pairwise coprime with $\:n=n_{1}n_{2}n_{3}n_{4}\:$, is $2^{s(n)}$. Therefore we get the following.
\begin{lemma}\label{lemma3}
We have
$$2^{s(n)}=\sum\limits_{\n}g(\n),$$
here the sum is over all factorizations,
\begin{center}
$n=\prod\limits_{i,j}n_{ij}$, $\quad n_{ij}$ are pairwise coprime.
\end{center}
where
$$g(\n)=\biggl(\frac{-1}{\alpha}\biggr)\biggl(\frac{3}{\beta}\biggr)\prod\limits_{i}4^{-\omega(n_{i0})}\prod\limits_{j\neq 0}4^{-\omega(n_{ij})}\prod\limits_{k\neq i,j}\prod\limits_{l}\biggl(\frac{n_{kl}}{n_{ij}}\biggr),$$
with
$$\alpha=n_{12}n_{14}n_{23}n_{21},\qquad \beta=n_{13}n_{14}n_{23}n_{24}.$$   
\end{lemma}
\section{averaging over \textbf{n}; linked variables}
\noindent From Lemma \ref{lemma3}, we are now in a position to estimate $2^{s(n)}$. Instead we take the sum of $2^{s(n)}$ over $S(X,h)$,
$$\sum\limits_{n\in S(X,h)}2^{s(n)}.$$ 
Following Heath-Brown's strategy, we take the sum over the 16 variables $n_{ij}$, subject to the following conditions: \\
each  $n_{ij}$ is square-free, they are pairwise coprime, each $n_{ij}$ is product of primes congruent to $1$ modulo $4$ and their product $n$ satisfies
\begin{center}
$n\leq X$, $\quad n\equiv h\pmod{24}$.
\end{center}
Let $N\in \mathbb{N}$ be such that $2^{N}\leq X<2^{N+1}$. Each variable $n_{ij}$ can be divided into intervals $(A_{ij},2A_{ij}]$, where $A_{ij}\in \left\{1,2,2^{2},\ldots,2^{N}\right\}$. We can observe that $N=O(\log
X)$. So we have $O(\log^{16} X)$ non-empty subsums, which is written as $S(\A)$, where $\A$ is the 16-tuple of numbers $A_{ij}$. At this point, we can assume that
\begin{equation}\label{eq:3}
1\ll \prod A_{ij}\ll X.
\end{equation}
We shall define the variables $n_{ij}$ and $n_{kl}$ as being `linked' if $i\neq k$ and precisely one of the conditions $l\neq 0,i$ or $j\neq 0,k$ holds. This implies that exactly one of the Jacobi symbols 
\begin{center}
$\left(\frac{n_{kl}}{n_{ij}}\right)$, $\left(\frac{n_{ij}}{n_{kl}}\right)$
\end{center}
appears in the expression for $g(\n)$. Without loss of generality, we may assume that the variables $n_{ij}$ and $n_{kl}$ are linked and the first between the two of the above Jacobi symbols appears in the expression for $g(\n)$. Then $g(\n)$ can be written in the form
$$g(\n)=\left(\frac{n_{kl}}{n_{ij}}\right)a(n_{ij})b(n_{kl}),$$
where the function $a(n_{ij})$ may depend on all the other variables $n_{uv}$ as well as $n_{ij}$, but is independent of $n_{kl}$, and the function $b(n_{kl})$ may depend on all the variables $n_{uv}$ as well as $n_{kl}$, but is independent of $n_{ij}$. Moreover
$$|a(n_{ij})|,\:|b(n_{kl})|\leq 1.$$
Now one can get the following inequality for $S(\A)$,
$$|S(\A)|\leq \sum\limits_{n_{uv}}|\sum\limits_{n_{ij},n_{kl}}\left(\frac{n_{kl}}{n_{ij}}\right)a(n_{ij})b(n_{kl})|.$$
We may re-express the condition of $(n_{ij}, n_{uv})=1$ and $(n_{kl}, n_{uv})=1$ for each $n_{uv}$, by 
 by taking the functions $a$ and $b$ to vanish at appropriate values. Moreover if $(n_{ij},n_{kl}) \neq 1$, then $\left(\frac{n_{kl}}{n_{ij}}\right)$ is zero. 
 With the above considerations, the inner sum is subject to the conditions that $n_{ij}$ and $n_{kl}$ are square-free and satisfy
\begin{center}
$n_{ij}n_{kl}\equiv h^{\prime} \pmod{24}$, $\quad n_{ij}n_{kl}\leq X^{\prime}$, $\quad n_{ij},n_{kl}\equiv 1\pmod{4}$,
\end{center}
\noindent where $h^{\prime}$ and $X^{\prime}$  depend on the other variables $n_{uv}$ and $(h^{\prime},24)=1$. We now use the following lemma.
\begin{lemma}\label{lemma4}
Let $a_{m}$, $b_{n}$ be complex numbers of modulus at most $1$. Let $h$ be an integer such that $(h,24)=1$ and let $M,N,X\gg 1$. Then 
$$\sum\limits_{m,n}\left(\frac{n}{m}\right)a_{m}b_{n}\ll MN\left\{\min(M,N)\right\}^{-\frac{1}{32}},$$
uniformly in $X$, where the sum is for square-free $m$, $n$ satisfying $M<m\leq 2M$, $N<n\leq 2N$, $m\equiv 1\pmod{4}$, $n\equiv 1\pmod{4}$, $mn\leq X$ and $mn\equiv h\pmod{24}$.
\end{lemma}
\noindent The proof is immediate using Lemma 4, in \cite{HeathBrown}. Although Lemma 4, in \cite{HeathBrown} is stated with modulo $8$, one can check that the same proof also works for modulo $24$.
Now it follows that
$$S(\A)\ll (\prod\limits_{uv}A_{uv})A_{ij}A_{kl}\left\{\min(A_{ij},A_{kl})\right\}^{-\frac{1}{32}}\ll X\left\{\min(A_{ij},A_{kl})\right\}^{-\frac{1}{32}},$$
by (\ref{eq:3}), and this leads to the following conclusion.
\begin{lemma}\label{lemma5}
$$S(\A)\ll X(\log X)^{-17},$$
if there exists a pair of linked variables $n_{ij}$, $n_{kl}$ with
$A_{ij},\: A_{kl}\geq \log^{544} X$.
\end{lemma}
\noindent Let's analyze the case where  $A_{ij}\geq \log^{544} X$ and  $A_{kl}<\log^{544} X$ for every variable $n_{kl}$ which is linked to $n_{ij}$. Where $n^{\prime}$ is the product of these variables $n_{kl}$. By using the quadratic reciprocity law,  $g(\n)$ can be written in the form
$$g(\n)=4^{-\omega(n_{ij})}\left(\frac{n_{ij}}{n^{\prime}}\right)\chi (n_{ij})c,$$
where $\chi$ is a Dirichlet character modulo $24$. $\chi$ may depend on the variables $n_{uv}$ other than $n_{ij}$, and the other factor $c$ is independent of $n_{ij}$ and satisfies $|c|\leq 1$. We get that
\begin{equation}\label{eq:4}
|S(\A)|\leq \sum\limits_{n_{uv}}| \sum\limits_{n_{ij}} 4^{-\omega(n_{ij})}\left(\frac{n_{ij}}{n^{\prime}}\right)\chi (n_{ij})|,
\end{equation}
where the inner sum is subject to the following conditions:\\
$n_{ij}$ is square-free, $n_{ij}\equiv 1\pmod{4}$, $(n_{ij}, n_{uv})=1$ for all the other variables $n_{uv}$, and $n_{ij}\equiv h^{\prime}\pmod{24}$, $A_{ij}<n_{ij}\leq \min(2A_{ij},X^{\prime})$.
Where $h^{\prime}$ and $X^{\prime}$ depend on the other variables $n_{uv}$ and $(h^{\prime},24)=1$. We now apply the following lemma.
\begin{lemma}\label{lemma6}
Let $N>0$ be given. Then for arbitrary positive integers $q$, $r$ and any non-principal character $\chi\pmod{q}$, we have
$$\sum\limits_{n\leq x, n\equiv 1 \text{mod}(4), (n,r)=1} \mu^2(n)4^{-\omega(n)}\chi (n)\ll xd(r)\exp(-c\sqrt{\log x}),$$
with a positive constant $c=c_{N}$, uniformly for $q\leq \frac{1}{4}\log^N x$.
\end{lemma}
\noindent The proof is immediate using Lemma 6, in \cite{HeathBrown}.
\noindent  By using the orthogonality principle of Dirichlet characters modulo $24$, we can express the condition $n_{ij}\equiv h^{\prime}\pmod{24}$ as follows: 
$$\frac{1}{8}\sum_{\psi\:(\text{mod}\:24)}\psi (n_{ij}) \overline{\psi (h^{\prime})}.$$
Therefore we insert a factor of
$$\frac{1}{8}\sum_{\psi\:(\text{mod}\:24)}\psi (n_{ij}) \overline{\psi (h^{\prime})},$$
to remove the condition $n_{ij}\equiv h^{\prime}\pmod{24}$.
Taking
\begin{center}
$q=24n^{\prime}\ll (\log^{544} X)^{15}$ and $r=\prod n_{uv}$,
\end{center}
we conclude that
$$S(\A)\ll A_{ij}\exp(-c\sqrt{\log A_{ij}})\sum\limits_{n_{uv}} d(r),$$
providing that $n^{\prime}\neq 1$. Since the variables $n_{uv}$ are pairwise coprime, we have $d(r)=\prod d(n_{uv})$. Moreover for each variable $n_{kl}$ we have
$$\sum\limits_{n_{uv}}d(n_{uv})\ll A_{kl}\log A_{kl}\ll A_{kl}\log X,$$
hence by (\ref{eq:3}) we have
$$S(\A)\ll X(\log X)^{15}\exp(-c\sqrt{\log A_{ij}}),$$
providing that $n^{\prime}\neq 1$. As a consequence of the above discussion we get
\begin{lemma}\label{lemma7}
There exists an absolute constant $\kappa>0$ such that 
$$S(\A)\ll X(\log X)^{-17},$$  
if there exists a pair of linked variables $n_{ij}$ and $n_{kl}$ with
\begin{equation}\label{eq:5}
A_{ij}\geq \exp\left\{\kappa (\log \log X)^{2}\right\}
\end{equation} 
and $n_{kl}>1$.
\end{lemma}
\noindent We will now find an asymptotic bound to deal with the case where at most three of the variables $n_{ij}$ take value in $(A_{ij},2A_{ij}]$ satisfying (\ref{eq:5}).\\

\noindent Define $C=\exp\left\{\kappa(\log \log X)^{2}\right\}$.
\begin{lemma}\label{lemma8}
We have 
$$\sum\limits_{\A}|S(\A)|\ll X(\log X)^{-\frac{5}{8}}(\log \log X)^{8},$$
where $\A$ denotes all sets in which either there exists at most three elements $A_{ij}\geq C$ or there exists linked variables $n_{ij}$ and $n_{kl}$ with $A_{ij}\geq C$ and $n_{kl}>1$.
\end{lemma}
\begin{proof}
Let's take that $C$ is some power of $2$. Suppose $\sum^{\prime}$ denotes that at most three of the $A_{ij}$ satisfy $A_{ij}\geq C$, then we have 
$$\sum_{A_{ij}}{}^{'}|S(\A)|\leq \sum_{d_{1},\ldots ,d_{16}\leq X}4^{-\omega(d_{1})}\ldots 4^{-\omega(d_{16})},$$
where $d_{i}$ is square-free, pairwise coprime, product of primes congruent to $1$ modulo $4$ and at most three of the $d_{i}$ have $d_{i}\geq 2C$. Let us write
\begin{center}
$m=\prod\limits_{d_{i}<2C}d_{i}$, $\quad n=\prod\limits_{d_{i}\geq 2C}d_{i}$,
\end{center}
therefore $m\leq (2C)^{16}$ and $n\leq \frac{X}{m}$. Moreover we observe that each value of $m$ can appear at most $16^{\omega(m)}$ times, and each value of $n$ can appear atmost $\binom{16}{3}3^{\omega(n)}$ times.\\

\noindent As a consequence we get that
$$\sum_{A_{ij}}{}^{'}|S(\A)|\ll \sum\limits_{m}4^{\omega(m)}\sum\limits_{n}\left(\frac{3}{4}\right)^{\omega(n)}.$$ 
The inner sum is restricted to the conditions $n\leq \frac{X}{m}$, n is product primes congruent to $1$ modulo $4$.
We now use the following bound
\begin{equation}\label{eq:6}
\sum\limits_{n\leq N}\gamma^{\omega(n)}\ll N(\log N)^{\frac{\gamma}{2}-1},
\end{equation}
which is true for any fixed $\gamma>0$. The sum is over $n\leq N$, $n$ is product of primes congruent to $1$ modulo $4$. As
$$\frac{X}{m}\gg XC^{-16}\gg X^{\frac{1}{2}},$$
we have $\log \frac{X}{m}\gg \log X$, and we get
\\

$$\sum_{A_{ij}}{}^{'}|S(\A)|\ll X(\log X)^{\frac{3} {8}-1}\sum\limits_{m}4^{\omega(m)}m^{-1} = X(\log X)^{-\frac{5} {8}}\sum\limits_{m}4^{\omega(m)}m^{-1} $$

Applying the partial summation and (\ref{eq:6}) we obtain
$$\sum\limits_{m\leq M}4^{\omega(m)}m^{-1}\ll \log^{4}M.$$
Hence we have
$$\sum_{A_{ij}}{}^{'}|S(\A)|\ll X(\log X)^{-\frac{5}{8}}\log^{4}C\ll  X(\log X)^{-\frac{5}{8}}(\log \log X)^{8}.$$
Now applying lemma \ref{lemma7}, we get the desired result.
\end{proof}
\section{averaging over \textbf{n}; characters modulo 24}
\noindent In this section our main goal is to find those sums $S(\A)$ that are not covered by Lemma \ref{lemma8}. These occur when we have four or more elements $A_{ij}\geq C$. Without loss of generality we can assume that $A_{10}\geq C$ and $A_{20}\geq C$. Therefore we have
$$n_{13}=n_{14}=n_{23}=n_{24}=n_{31}=n_{32}=n_{34}=n_{41}=n_{42}=n_{43}=1,$$
as they are linked to $A_{10}$ or $A_{20}$. Therefore two or more among $A_{12}$, $A_{21}$, $A_{30}$ and $A_{40}$ must be at least $C$. If $A_{12}\geq C$, then $n_{30}=n_{40}=1$, since they are linked to $n_{12}$. Same conclusion holds if $A_{21}\geq C$. On the other hand, if $A_{30}$ or $A_{40}$ is at least $C$, then we will have $n_{12}=n_{21}=1$. Therefore we deduce that if $A_{10},\: A_{20}\geq C$, then we must have either $A_{12},\: A_{21}\geq C$ and the remaining variables are $1$, or $A_{30},\: A_{40}\geq C$ and the remaining variables are $1$. We can show that a similar conclusion holds if $A_{i0},\: A_{j0}\geq C$.\\

\noindent Now we suppose that there is only one element $A_{i0}$ such that $A_{i0}\geq C$. Without loss of generality assume that this is $A_{10}$. Then we have
$$n_{23}=n_{24}=n_{32}=n_{34}=n_{42}=n_{43}=1,$$
since they are linked to $A_{10}$. If $A_{12},\: A_{21}\geq C$, then also we have
$$n_{13}=n_{14}=n_{30}=n_{31}=n_{40}=n_{41}=1,$$
which contradicts the fact that there are four or more elements $A_{ij}\geq C$. Same argument applies if $A_{13},\: A_{31}\geq C$ or $A_{14},\: A_{41}\geq C$. Therfore we  have either $A_{12},\: A_{13},\: A_{14}\geq C$, or $A_{21},\: A_{31},\: A_{41}\geq C$. We note that in both the cases the remaining variables $n_{ij}=1$, as each one is linked to a variable $n_{kl}$ with $A_{kl}\geq C$.\\

\noindent Finally we are left with the case in which all of the variables $A_{i0}<C$. Suppose $A_{12},\: A_{13}\geq C$, then we have
$$n_{20}=n_{21}=n_{23}=n_{24}=n_{30}=n_{31}=n_{34}=n_{40}=n_{41}=1.$$
If also $A_{14}\geq C$, then $n_{42}=n_{43}=1$. which contradicts the fact that there are four or more elements $A_{ij}\geq C$.Hence we have $A_{42},\: A_{43}\geq C$, and all the remaining variables $n_{ij}$ will be $1$. An analogous argument holds if $A_{ij},\: A_{ik}\geq C$ with $i,\:j,\:k$ distinct. So we are left with the possibility that $A_{ij}\geq C$ for four different values of $i$. Without loss of generality, we can assume that one of them is $A_{12}$, then we have
$$n_{20}=n_{23}=n_{24}=1,$$
since they are linked to $n_{12}$, and so $A_{21}\geq C$. $n_{34}$ and $n_{43}$ are the only variables which are not linked  neither  to $n_{12},$ nor linked to $n_{21}$. It implies that $A_{34},\: A_{43}\geq C$, and the remaining variables $n_{ij}=1$.\\

\noindent The upshot of the preceding discussion is the following lemma.
\begin{lemma}\label{lemma9}
A sum $S(\A)$ which is not covered by Lemma \ref{lemma8} will have exactly four elements $A_{ij}\geq C$, and the remaining variables $n_{kl}$ will be $1$. The possibility of the indices $ij$ are 
\begin{align*}
10,20,30,40&  &i0,j0,ij,ji&  &i0,ij,ik,il\\
i0,ji,ki,li&  &ij,ik,lj,lk&  &ij,ji,kl,lk\\
\end{align*}
\noindent where $i$, $j$, $k$, $l$ denote different non-zero indices.   
\end{lemma}
\noindent In view of lemma \ref{lemma9}, we are left with these $24$ types of sum. We shall re-label the variables $n_{ij}$ which appear non-trivially as $n_{1},\ldots,n_{4}$
and write $N_{1},\ldots,N_{4}$ for corresponding $A_{ij}$. We shall define the variables $n_{i}$ and $n_{j}$ are `joined' if both the Jacobi symbols
\begin{center}
$\left(\frac{n_{i}}{n_{j}}\right)$, $\left(\frac{n_{j}}{n_{i}}\right)$
\end{center}
appear in the definition of $g(\n)$. Thus $n_{ij}$ and $n_{kl}$ are joined if $i\neq k$ and $j,l\neq i,k,0$. If two variables are not joined we shall say that they are `independent'. By abuse of terminology, we shall also refer to the indices $ij$ and $kl$ as being `joined' or `independent', as appropriate. For each $\A$ in Lemma \ref{lemma9}, $S(\A)$ can be written in the form
\begin{equation}\label{eq:7}
\sum\limits_{n_{1},\ldots,n_{4}}\chi_{1}(n_{1})\chi_{2}(n_{2})\chi_{3}(n_{3})\chi_{4}(n_{4})PQ,\quad Q=4^{-\omega(n_{1}\ldots n_{4})}
\end{equation}
Where the variables $n_{i}$ are square-free, pairwise coprime, $n_{i}\equiv 1\pmod{4}$and satisfy $N_{i}<n_{i}\leq 2N_{i}$. Here the characters $\chi_{i}$ are to modulus $24$, and are arising from $\left(\frac{-1}{\alpha}\right)$ and $\left(\frac{3}{\beta}\right)$ in the definition of $g(\n)$ (as in Lemma \ref{lemma3}) and $P$ is the product of the following expressions
$$(-1)^{\frac{(n_{i}-1)(n_{j}-1)}{4}},$$
one for each pair of joined variables.\\
\noindent We can observe that $P=1$. As each $n_{i}\equiv 1\pmod{4}$. Moreover we observe that $\left(\frac{-1}{\alpha}\right)=1$  \\

\noindent  We will now proceed by applying the following two lemmas. 
\begin{lemma}\label{lemma10}(Lemma 10, \cite{HeathBrown})
Let $X>0$ and $M,N\geq C>0$ be given. Then for an arbitrary positive integer $r$, any integer h such that $(h,24)=1$, and any distinct characters $\chi_{1},\chi_{2}\pmod{24}$, we have
$$\sum\limits_{m,n}\mu^{2}(m)\mu^{2}(n)4^{-\omega(m)-\omega(n)}\chi_{1}(m)\chi_{2}(n)\ll d(r)X\exp(-c\sqrt{\log C})\log X,$$
for some positive absolute constant $c$, where the sum is over coprime variables satisfying the conditions
$$M<m\leq 2M,\quad N<n\leq 2N,\quad mn\leq X,\quad mn\equiv h\pmod{24},\quad (mn,r)=1.$$
\end{lemma}
\noindent The proof of Lemma 10 in \cite{HeathBrown}, uses character modulo $8$. However one can check that the same proof also works for character modulo $24$.
\begin{lemma}\label{characterlemma}
Let $X>0$ and $M,N\geq C>0$ be given. Then for an arbitrary positive integer $r$, any integer h such that $(h,24)=1$, and for any non-principal character $\chi \pmod{24}$, we have
$$\sum\limits_{m,n}\mu^{2}(m)\mu^{2}(n)4^{-\omega(m)-\omega(n)}\chi(m)\chi(n)\ll d(r)X\exp(-c\sqrt{\log C})\log X,$$
for some positive absolute constant $c$, where the sum is over coprime variables satisfying the conditions
$$M<m\leq 2M,\quad N<n\leq 2N,\quad mn\leq X,\quad mn\equiv h\pmod{24},\quad (mn,r)=1.$$
\end{lemma}
\noindent The proof is straightforward using Lemma 6, in \cite{HeathBrown}.\\

\begin{center}
$10,20,30,40$\\
$30,40,34,43$\\
$30,31,32,34$\\
$40,41,42,43$\\
$31,32,41,42$\\
$12,21,34,43$\\
$10,21,31,41$\\
$20,12,32,42$\\
$10,20,12,21$\\
\end{center}

\noindent Applying Lemma \ref{lemma10} and Lemma
\ref{characterlemma}, we get that the sums $S(\A)$ except for the above mentioned indices are all $O(X(\log X)^{-17})$, since one can take the constant $\kappa$ in Lemma \ref{lemma7} sufficiently large. Therefore the total contribution of these sums is $O(X(\log X)^{-1})$, which is satisfactory for us. Therefore we deduce as follows.
\begin{lemma}\label{lemma11}
$$\sum\limits_{\A}S(\A)\ll X(\log X)^{-\frac{5}{8}}(\log \log X)^{8},$$
where $\A$ denotes all sets other than those corresponding to the indices
\begin{center}
$10,20,30,40;\quad 30,40,34,43;\quad 30,31,32,34;\quad 40,41,42,43;\quad 10,20,12,21;$
\end{center}
\begin{center}
$31,32,41,42;\quad 12,21,34,43;\quad 10,21,31,41;\quad 20,12,32,42$.
\end{center}
\end{lemma}
\begin{remark}
 This statement differs from the corresponding result in \cite{HeathBrown}, as we have $9$ exceptional cases instead of $4$ exceptional cases in \cite{HeathBrown}.    
\end{remark}
\section{The main terms}\label{sec:6}
\noindent Initially we observe that the sums corresponding to the indices $10,20,30,40;\quad 30,40,34,43;\quad 30,31,32,34;\quad 40,41,42,43;$ $31,32,41,42;\quad 10,20,12,21;\quad 12,21,34,43;\quad 10,21,31,41;\quad 20,12,32,42;\quad$  the function $g(\n)$ becomes to $4^{-\omega(n)}$, where $n$ is the product of the variables $n_{ij}$. The contribution of all sums with indices $10,20,30,40$ and $A_{i0}\geq C$, is therefore 
$$\sum\limits_{n_{i0}}4^{-\omega(n)}$$
where the sum is restricted to the conditions
\begin{center}
$n_{i0}>C,\quad n\leq X,\quad n\equiv h\pmod{24}$, $n$ square-free, every prime divisor of $n$ is congruent to $1$ modulo $4$.
\end{center}
We can remove the condition $n_{i0}>C$ with an error
\begin{align*}
&\ll \sum\limits_{abcd\leq X,a\leq C}\mu^{2}(abcd)4^{-\omega(a)-\omega(b)-\omega(c)-\omega(d)}\\
&=\sum\limits_{ae\leq X,a\leq X}\mu^{2}(ae)4^{-\omega(a)}(\frac{3}{4})^{\omega(e)}\\
&\ll \sum\limits_{a\leq C}4^{-\omega(a)}\sum\limits_{e\leq \frac{X}{a}}(\frac{3}{4})^{\omega(e)}
\end{align*}
The inner sum is restricted to the conditions $e\leq \frac{X}{a}$,
$e$ is product of primes congruent to $1$ modulo $4$. Therefore by applying (\ref{eq:6}) we obtain

$$\sum\limits_{a\leq C}4^{-\omega(a)}\sum\limits_{e\leq \frac{X}{a}}(\frac{3}{4})^{\omega(e)} \ll X(\log X)^{-\frac{5}{8}}\sum\limits_{a\leq C}4^{-\omega(a)}a^{-1} \ll X(\log X)^{-\frac{1}{4}}(\log \log X)^{2}$$
Since $n$ is square-free it factorizes as $n_{10}n_{20}n_{30}n_{40}$, where $n_{i0}$ are pairwise coprime in exactly $4^{\omega(n)}$ different ways. We therefore get
\begin{equation}\label{eq:10}
\sum\limits_{n_{i0}>C}4^{-\omega(n)}=\#\:S(X,h)+O(X(\log X)^{-\frac{5}{8}}(\log \log X)^{2}).
\end{equation}

 \noindent Therefore Theorem \ref{maintheorem} follows.

\section{Applications}\label{applications}
\noindent In this section, we collect some interesting applications of the Theorem \ref{maintheorem}. In the congruent number problem case, Monsky has proved the parity of $s(n)$ according to the congruence class modulo $8$. In the case of $\frac{\pi}{3}$-congruent number scenario, Wei-Guo have proved the following theorem:
\begin{theorem}\label{Theorem 1}(Theorem 1.3, \cite{WeiGuo})
n is a square-free, positive integer. Then we have \\
$s(n)$ is even if $n\equiv 1, 2, 3, 5, 7, 9, 14, 15, 19 \pmod{24}$, $s(n)$ is odd if $n\equiv 6, 10, 11, 13, 17, 18, 21, 22, 23 \pmod{24}$.
\end{theorem}
\noindent When $n\equiv 1,5,7,19 \pmod{24}$, $s(n)$ is even by Theorem \ref{Theorem 1}. So that
$$s(n)\leq \frac{1}{2}2^{s(n)}$$
Therefore we deduce the following average bounds for $s(n)$.
\begin{corollary}
Then for $h=5$ we have
$$\sum\limits_{n\in S(X,h)}s(n)\leq \frac{9}{2}\#\, S(X,h)+ O(X(\log X)^{-\frac{5}{8}}(\log \log X)^{8}).$$
\end{corollary}
\noindent This has an interesting consequence for the density of $s(n)=0$ or $2$.
\begin{corollary}\label{s(n)=0or2}
The density of $s(n)=0$ or $s(n)=2$ among all square-free $n\equiv 5 \pmod{24}$ having all prime divisors congruent to $1$ modulo $4$ is at least $\frac{7}{16}$.
\end{corollary}
\begin{proof}
We define $S(0)=\left\{n\in S(X,h):s(n)=0\right\}$, $S(2)=\left\{n\in S(X,h):s(n)=2\right\}$ and $S(>2)=\left\{n\in S(X,h):s(n)>2\right\}$. By Theorem \ref{Theorem 1}, we know that $s(n)$ is even for $n\equiv 1,5,7,19 \pmod{24}$. Now by using Corollary \ref{averagerank}, we get,
$$\#\:S(0)+4\:\#\:S(2)+16\:\#\:S(>2)\leq 9\#\, S(X,h)+ O(X(\log X)^{-\frac{5}{8}}(\log \log X)^{8}).$$
Dividing by $\#\:S(X,h)$ in the both sides and letting $X\to\infty$ we get,
$$\lim_{X\to\infty}\frac{\#\:S(0)}{\#\:S(X,h)}+\lim_{X\to\infty}4\frac{\#\:S(2)}{\#\:S(X,h)}+\lim_{X\to\infty}16\frac{\#\:S(>2)}{\#\:S(X,h)}\leq 9.$$
On the other hand, we have
$$\lim_{X\to\infty}\frac{\#\:S(0)}{\#\:S(X,h)}+\lim_{X\to\infty}\frac{\#\:S(2)}{\#\:S(X,h)}+\lim_{X\to\infty}\frac{\#\:S(>2)}{\#\:S(X,h)}=1.$$
From the above expressions, we therefore deduce that
$$\lim_{X\to\infty}\frac{\#\:S(0)}{\#\:S(X,h)}+\lim_{X\to\infty}\frac{\#\:S(2)}{\#\:S(X,h)}\geq \frac{7}{16}.$$
Hence the corollary follows.
\end{proof}
\begin{remark}
We would like to emphasize that the positive density of $s(n)=0$ or $s(n)=2$ among of all square-free $n\equiv 13\pmod{24}$ having all prime divisors congruent to $1$ modulo $4$ is an unconditional result. 
\end{remark}

\begin{corollary}
Then for $h=5$ we have
$$\sum\limits_{n\in S(X,h)}r(n)\leq \frac{9}{2}\#\, S(X,h)+ O(X(\log X)^{-\frac{5}{8}}(\log \log X)^{8}).$$
\end{corollary}
\begin{remark}
Assuming that $r(n)$ and $s(n)$ always have the same parity, we can deduce that the density of $r(n)=0$ or $r(n)=2$ among all square-free $n\equiv 5 \pmod{24}$ having all prime divisors congruent to $1$ modulo $4$ is at least $\frac{7}{16}$.
\end{remark}
\noindent When $n\equiv 11,13,17,23 \pmod{24}$, $s(n)$ is odd by Theorem \ref{Theorem 1}. So we have
$$s(n)\leq \frac{1}{3}(2^{s(n)}+1).$$
In a similar way as above, we deduce the following average bounds for $s(n)$.
\begin{corollary}\label{cor1}
Then for $h=13$ we have
$$\sum\limits_{n\in S(X,h)}s(n)\leq \frac{10}{3}\#\, S(X,h)+ O(X(\log X)^{-\frac{5}{8}}(\log \log X)^{8}).$$
\end{corollary}
\noindent This also has an interesting consequence for the density of $s(n) =1$ or $3$.
\begin{corollary}\label{s(n)=1}
The density of $s(n)=1$ or $s(n)=3$ among of all square-free $n\equiv 13\pmod{24}$ having all prime divisors congruent to $1$ modulo $4$ is at least $\frac{1}{3}$.   
\end{corollary}
\begin{proof}
We define $S(1)=\left\{n\in S(X,h):s(n)=1\right\}$ , $S(3)=\left\{n\in S(X,h):s(n)=3\right\}$ and $S(>3)=\left\{n\in S(X,h):s(n)>3\right\}$. By Theorem \ref{Theorem 1}, we know that $s(n)$ is odd for  $n\equiv 13 \pmod{24}$. Now by Corollary \ref{cor1}, we have,
$$\#\:S(1)+3\:\#\:S(3)+5\:\#\:S(>3)\leq \frac{10}{3}\#\, S(X,h)+ O(X(\log X)^{-\frac{5}{8}}(\log \log X)^{8}).$$
Dividing by $\#\:S(X,h)$ in the both sides and letting $X\to\infty$ we get, 
$$\lim_{X\to\infty}\frac{\#\:S(1)}{\#\:S(X,h)}+\lim_{X\to\infty}3\frac{\#\:S(3)}{\#\:S(X,h)}+\lim_{X\to\infty}5\frac{\#\:S(>3)}{\#\:S(X,h)}\leq \frac{10}{3}.$$
On the other hand, we have,
$$\lim_{X\to\infty}\frac{\#\:S(1)}{\#\:S(X,h)}+\lim_{X\to\infty}\frac{\#\:S(3)}{\#\:S(X,h)}+\lim_{X\to\infty}\frac{\#\:S(>3)}{\#\:S(X,h)}=1.$$
From the above expressions, we therefore deduce that
$$\lim_{X\to\infty}\frac{\#\:S(1)}{\#\:S(X,h)}+\lim_{X\to\infty}\frac{\#\:S(3)}{\#\:S(X,h)}\geq \frac{1}{3}.$$
\end{proof}
\begin{remark}
We would like to emphasize that the positive density of $s(n)=1$ or $s(n)=3$ among of all square-free $n\equiv 13\pmod{24}$ having all prime divisors congruent to $1$ modulo $4$ is an unconditional result. 
\end{remark}

\begin{corollary}
Then for $h=13$ we have
$$\sum\limits_{n\in S(X,h)}r(n)\leq \frac{10}{3}\,\#\, S(X,h)+ O(X(\log X)^{-\frac{5}{8}}(\log \log X)^{8}).$$
\end{corollary}
\begin{remark} Assuming that $r(n)$ and $s(n)$ always have the same parity. We can deduce that the density of $r(n)=1$ or $r(n)=3$ among all square-free $n\equiv 13\pmod{24}$ having all prime divisors congruent to $1$ modulo $4$ is at least $\frac{1}{3}$. This implies that the density of $\frac{\pi}{3}$-congruent number among all square-free $n\equiv 13\pmod{24}$ having all prime divisors congruent to $1$ modulo $4$ is at least $\frac{1}{3}$.
\end{remark}

\bibliographystyle{alpha}

\end{document}